\documentclass[10pt]{amsart}

\usepackage{amsmath, amsthm, amscd, amsfonts, amssymb, graphicx, color}
\usepackage{latexsym}
\usepackage{amscd}
\usepackage{amsthm}
\usepackage{amssymb}
\usepackage{enumerate}
\usepackage{mathabx}
\usepackage[utf8]{inputenc}
\usepackage[T1]{fontenc}

\usepackage{lmodern}

\theoremstyle{plain}
\newtheorem{theorem}{Theorem}
\newtheorem{corollary}[theorem]{Corollary}
\newtheorem{lemma}[theorem]{Lemma}

\theoremstyle{remark}

\newtheorem{example}[theorem]{Example}

\newcommand{\R}{\mathbb{R}}
\newcommand{\N}{\mathbb{N}}

\begin{document}

\title{On the closure of translation-dilation invariant linear spaces of polynomials }
\author{J.~M.~Almira and L.~Sz\'ekelyhidi}

\subjclass[2010]{Primary 41A10, 40A30; Secondary 39A70.}

\keywords{Polynomials, Translation Invariance, Dilation Invariance, Pointwise Convergence, Difference Operators.}
\let\thefootnote\relax\footnotetext{Second author was supported by the Hungarian National Foundation for Scientific Research (OTKA),   Grant No. K111651.}

\address{Departamento de Matem\'{a}ticas, Universidad de Ja\'{e}n, E.P.S. Linares,  Campus Cient\'{\i}fico Tecnol\'{o}gico de Linares, Cintur\'{o}n Sur s/n,  23700 Linares, Spain}
\email{jmalmira@ujaen.es}
\address{Institute of Mathematics, University of Debrecen, Egyetem t\'er 1, 4032 Debrecen, Hungary --- Department of Mathematics, University of Botswana, 4775 Notwane 
Rd. Gaborone, Botswana }
\email{lszekelyhidi@gmail.com}

\maketitle
\begin{abstract}
Assume that a linear space of real polynomials in $d$ variables is given which is translation and dilation invariant. We show that if a sequence in this space converges pointwise to a polynomial, then the limit polynomial belongs to the space, too.
\end{abstract}

\markboth{J.~M.~Almira, L.~Sz\'ekelyhidi}{On the closure of translation-dilation invariant linear spaces of polynomials}

\section{Introduction}
At the 49th International Symposium on Functional Equations in Graz, Mariatrost, Austria, 2011 and later at the 14th International Conference on Functional Equations and Inequalities in B\k{e}dlewo, Poland, 2011 the second author proposed the following problem: Assume that $V$ is a linear space of real polynomials in $n$ variables which is translation invariant. Suppose moreover that the sequence $(p_n)$ in $V$ converges pointwise to a polynomial $p$. Is it true that $p$ is in $V$? Despite several efforts of different researchers this question has still remained open. In this note we solve the problem in the positive for the special case when $V$ is a translation-dilation invariant linear space of polynomials.
 
\section{Translation and dilation invariant subspaces} 
In this paper $\R$ denotes the set of real numbers and $\R[x]$ denotes the polynomial ring in the variables $x=(x_1,x_2,\dots,x_d)$, where $d$ is a positive integer. We shall use the following standard notation. For every $a=(a_1,a_2,\dots,a_d)$ and $x=(x_1,x_2,\dots,x_d)$ in $\R^d$ we let $a\cdot x=(a_1 x_1,a_2 x_2,\dots,a_d x_d)$. For every multi-index $\alpha=(\alpha_1,\alpha_2\dots,\alpha_d)$ in $\N^d$ we write $|\alpha|=\alpha_1+\alpha_2+\dots+\alpha_d$ and $x^{\alpha}=x_1^{\alpha_1} x_2^{\alpha_2}\dots x_d^{\alpha_d}$, where we use the convention $0^0=1$. For convenience the monomial $x\mapsto x^{\alpha}$ will be denoted by $q_{\alpha}$ for each multi-index $\alpha$. Also we introduce the notation $\alpha!=\alpha_1! \alpha_2!\dots \alpha_d!$ for every multi-index $\alpha $ and we shall use the partial order $ \alpha\leq \beta$ componentwise. We shall use the notation $\widehat{\N}$ for the set $\N\cup \{\infty\}$ and we extend the order relation $\leq$ from $\N$ to $\widehat{\N}$, as well as the partial order relation $\leq$ from $\N^d$ to $\widehat{\N}^d$ in the obvious manner. The initial section corresponding to $\alpha$ in $\widehat{\N}$ with respect to the partial order $\leq$ will be denoted by $[\alpha]$. In other words
$$
[\alpha]=\{\beta\in \widehat{\N}^d:\, \beta\leq \alpha\}.
$$
\vskip.2cm 

We also introduce the spaces  $B_{[\alpha]}=\text{span}\{x^\beta:\beta\in [\alpha]\}$. In addition we shall use the standard notation for partial differential operators
$$
P(\partial)=\sum_{\alpha} c_{\alpha} \partial_1^{\alpha_1}\partial_2^{\alpha_2}\dots \partial_d^{\alpha_d}
$$
whenever $P:\R^d\to\R$ is the polynomial $P(x)=\sum_{\alpha} c_{\alpha} x_1^{\alpha_1}x_2^{\alpha_2}\dots x_d^{\alpha_d}$. In particular, we have $q_{\alpha}(\partial)=\partial^{\alpha}$.
\vskip.1cm

Also we shall use difference operators with multi-index notation. For every $k=1,2,\dots,d$  the symbol $\Delta_k$ denotes the partial difference operator on polynomials acting on the $k$-th variable with increment $1$, that is
$$
\Delta_k p(x)=p(x+e_k)-p(x)
$$
for each $x$, where $e_k$ is the element in $\R^d$ whose $k$-th component is $1$, all the others are $0$. In addition, $e$ is the element whose all components are $1$. Then $\Delta$ is the vector difference operator defined by $\Delta=(\Delta_1,\Delta_2,\dots,\Delta_d)$. Using this notation we have
$$
\Delta^{\alpha} p=\Delta_1^{\alpha_1} \Delta_2^{\alpha_2}\dots \Delta_d^{\alpha_d}p
$$
for every multi-index $\alpha$. Then the meaning of $P(\Delta)$ is obvious for every polynomial $P$ in $d$ variables. In particular, we can write $q_{\alpha}(\Delta)=\Delta^{\alpha}$.
\vskip.2cm

Given a function $f:\R^d\to\R$ and $y$ in $\R^d$ we denote by $\tau_y$, resp. $\sigma_y$ the functions defined by
$$
\tau_yf(x)=f(x+y),\enskip \sigma_yf(x)=f(x\cdot y)
$$
whenever $y$ is in $\R^d$. We call $\tau_y$, resp. $\sigma_y$  {\it translation}, resp. {\it dilation} by $y$, further $\tau_yf$, resp. $\sigma_yf$ the {\it translate}, resp. the {\it dilate} of $f$ by $y$. A set $H$ of real functions on $\R^d$ is called {\it translation invariant}, resp. {\it dilation invariant}, if $\tau_yf$ is in $H$, resp. $\sigma_yf$ is in $H$ for each $f$ in $H$ and for every $y$ in $\R^d$. If $H$ is both translation and dilation invariant, then we call it {\it translation-dilation invariant}, or shortly a {\it TDI-set}. Given a function $f$ on $\R^d$ the intersection of all translation invariant, resp. dilation invariant, resp. translation-dilation invariant linear spaces including $f$ is denoted by $\tau(f)$, resp. $\sigma(f)$, resp $\tau\sigma(f)$. Clearly, these are linear spaces, and $\tau(f)$ is translation invariant, $\sigma(f)$ is dilation invariant, further $\tau\sigma(f)$ is translation-dilation invariant.  TDI-subspaces have been studied extensively and the following classification of closed TDI-subspaces has been proved in \cite{P}. 

\begin{theorem}\label{Pinkus} Every closed subspace of $\mathcal{C}(\mathbb{R}^d)$ which is invariant under all mappings  $S_{a,b}=\sigma_a\tau_b$ with arbitrary $a,b$ in $\mathbb{R}^d$ is the closure of the linear span of the union of the sets $B_{[n_k]}$ for a finite set of points $n_k$ in $\widehat{\N}^d$.
\end{theorem}

We note that here "closed" refers to the topology of the uniform convergence on compact sets.
\vskip.1cm

In this paper we study translation-dilation linear spaces of polynomials. We begin with some preliminary lemmas. The first one is a standard result in Algebra \cite[Chapter 5]{Lang}, but we include the proof for the sake of completeness.

\begin{lemma}\label{1}
Let $S$ be a set of different multi-indices in $\N^d$. Then the set of monomials $\{q_{\alpha}:\,\alpha\in S\}$ is linearly independent.
\end{lemma}

\begin{proof}
Indeed, if $c_{\alpha}$ is a real number for each $\alpha$ in $S$ such that we have
$$
\sum_{\alpha} c_{\alpha} q_{\alpha}(x)=0
$$
for each $x$ in $\R^d$, and $\alpha_0$ is any element in $S$, then we apply the differential operator $\partial^{\alpha_0}$ to the above equation. We obtain
$$
\sum_{\alpha\in S} c_{\alpha} \partial^{\alpha_0}q_{\alpha}(x)=\sum_{\alpha\geq \alpha_0, \alpha\in S} c_{\alpha} \frac{\alpha!}{\alpha_0!} x^{\alpha-\alpha_0}=0
$$
for each $x$ in $\R^d$. Now we substitute $x=0$. If for some $\alpha$ in $S$ $\alpha>\alpha_0$, then the corresponding term is zero, hence the above sum is equal to $c_{\alpha_0} \frac{\alpha_0!}{\alpha_0!}$ which implies $c_{\alpha_0}=0$.
\end{proof}

\begin{lemma}\label{2}
Let $p:\R^d\to\R$ be a polynomial. Then $\tau(p)$ is generated by all partial derivatives of $p$.
\end{lemma}

\begin{proof}
The statement follows immediately from the Taylor Formula
\begin{equation}\label{Taylor}
\tau_yp=\sum_{|\alpha|\leq \deg p} \frac{1}{\alpha!} \partial^{\alpha}p\, q_{\alpha}(y)
\end{equation}
which holds for each polynomial in $\R[x]$ and for every $x$ in $\R^d$.  Indeed, this formula shows that every translate of $p$ is a linear combination of partial derivatives \hbox{of $p$.} On the other hand, let $s$ denote the number of different multi-indices $\alpha$ with \hbox{$\alpha\leq \deg p$.} As the monomials $q_{\alpha}$ are linearly independent for different multi-indices $\alpha$ there exist elements $y_j$ for $j=1,2,\dots,s$ such that the quadratic matrix $\bigl(q_{\alpha}(y_j)\bigr)$ is regular (see, for example, \cite[Chapter 14]{AD}). Substituting $y_j$ for $y$ in the above equation we obtain a system of linear equations with regular matrix from which we can express $\partial^{\alpha}p$ as a linear combination of the translates $\tau_{y_j}p$ of $p$, hence all these partial derivatives belong to $\tau(p)$.
\end{proof}

\begin{lemma}\label{3}
Let $p:\R^d\to\R$ be a polynomial. Then $\sigma(p)$ is generated by all monomials of the form $\partial^{\alpha}p(0) x^{\alpha}$.
\end{lemma}

\begin{proof}
We use Taylor Formula \eqref{Taylor} again. We have for $y=0$
\begin{equation}\label{Taylor1}
p(x)=\sum_{|\alpha|\leq \deg p} \frac{1}{\alpha!} \partial^{\alpha}p(0) x^{\alpha}
\end{equation}
which means that it is enough to show that every monomial term on the right side is in $\sigma(p)$. In other words, we have to show that $\partial^{\alpha}p(0) q_{\alpha}$ is in $\sigma(p)$ for each $\alpha$ with $|\alpha|\leq \deg p$. By \eqref{Taylor1}, we have
\begin{equation}\label{Taylor2}
\sigma_{\lambda}p=\sum_{|\alpha|\leq \deg p} \frac{1}{\alpha!} q_{\alpha}(\lambda) \partial^{\alpha}p(0) q_{\alpha}
\end{equation}
for each $\lambda$ in $\R^d$. As the functions $\frac{1}{\alpha!} q_{\alpha} $ are linearly independent for different multi-indices $\alpha$ there exist elements $\lambda_j$ for $j=1,2,\dots,s$ such that the quadratic matrix $\bigl(\frac{1}{\alpha!}\,q_{\alpha}(\lambda_j)\bigr)$ is regular. Substituting $\lambda_j$ for $\lambda$ in the above equation we obtain a system of linear equations with regular matrix from which we can express $\partial^{\alpha}p(0) q_{\alpha}$ as a linear combination of the dilates $\sigma_{\lambda_j}p$ of $p$, hence all these monomials belong to $\sigma(p)$.
\end{proof}

\begin{corollary} \label{nuevo}
A linear space of real polynomials in several variables is dilation invariant if and only if it admits a basis formed by monomials. 
\end{corollary}

\begin{proof}
Obvious.
\end{proof}

\begin{corollary}\label{6}
Let $p:\R^d\to\R$ be a polynomial. Then $\tau\sigma(p)$ is generated by all monomials of the form $x^{\beta}$ such that, for a certain multi-index 
$\alpha$,    $\beta\leq\alpha$ and $ \partial^{\alpha}p(0)\neq 0$. In particular, if $\partial^{\alpha}p(0) \ne 0$, then $\tau\sigma(p)$ includes $x^{\beta}$ for each $\beta$ with $\beta\leq \alpha$.
\end{corollary}

\begin{proof}
Obvious, by Lemma \ref{2} and Lemma \ref{3}.
\end{proof}

We introduce the following notation: for a subset $H$ of $\mathbb{R}[x]$ we let
$$
\Omega_H=\{\alpha\in\mathbb{N}^d :x^{\alpha}\enskip\text{is in}\enskip H\}.
$$

Our main theorem follows.

\begin{theorem} \label{main}
If a sequence in a TDI-subspace in the polynomial ring $\R[x]$ converges pointwise to
a polynomial, then this polynomial belongs to the subspace, too.
\end{theorem}

\begin{proof} 
Let us first assume that $V$ is a TDI-subspace of  $\R[x]$ such that there exist natural numbers $N_1,\cdots, N_d$ satisfying
\[
\Omega_V=\bigcup_{k=1}^{d}Z_i,
\]
where  $Z_i=\{\alpha\in \mathbb{N}^d: \alpha_i\leq N_i\}$ for $i=1,2,\dots,d$.  
\vskip.2cm

Let the sequence $(p_n)_{n\in\N}$ of the TDI-subspace $V$ in $\R[x]$ converge pointwise to the polynomial $p$. Our assumption on $\Omega_V$ implies that,  for $n=0,1,\dots$ we can write $p_n$ in the following form
\[
 p_n(x_1,x_2,\dots,x_d)  = \sum_{k=1}^d\sum_{i=0}^{N_k}f_{n,k,i}(x_1,\dots,x_{k-1},x_{k+1},\dots, x_d)\,x_k^{i}
\]
with some polynomials $f_{n,k,i}$ arbitrary polynomials in $d-1$ variables.
\vskip.2cm

Proving by contradiction we assume that $p$ is not in $V$. We shall use the following notation: for $k=1,2,\dots,d$ let $e_k$ the element of $\R^d$ whose $k$-th component is $1$, all the others are $0$. 
\vskip.2cm

For each $x$ in $\R^d$ we have 
\[
\lim_{n\to\infty} \sum_{k=1}^d\sum_{i=0}^{N_k}f_{n,k,i}(x_1,\dots,x_{k-1},x_{k+1},\dots, x_d)x_k^{i}= p(x_1,x_2,\dots,x_d)
\] 
and $\partial^{\alpha} p(0)\ne 0$ for some $\alpha>N$. We apply the difference operator $\Delta^N$ on both sides of this equation. It is easy to see that
$$
\Delta^N p(x_1,x_2,\dots,x_d)=x_1 x_2\cdots x_d\cdot h(x_1,x_2,\dots,x_d)
$$ 
with some nonzero $h$ in $\R[x]$. On the other hand, on the left hand side we have 
$$
\lim_{n\to\infty}\sum_{k=1}^d\sum_{i=0}^{N_k} \Delta^N f_{n,k,i}(x_1,\dots,x_{k-1},x_{k+1},\dots, x_d)x_k^{i}= 
$$
$$
\lim_{n\to\infty}\sum_{k=1}^d F_{n,k}(x_1,\dots,x_{k-1},x_{k+1},\dots,x_d) 
$$
with some polynomials $F_{n,k}$ for $n=0,1,\dots$ and $k=1,2,\dots,d$. Now we substitute successively $x_j=1$ for $j=1,2,\dots,d$ and we let $n\to\infty$ to obtain

\begin{eqnarray*}
&\ &x_2\cdot x_3\cdots x_d h(1,x_2,\cdots,x_d) = \\ 
&\ &  \lim_{n\to\infty} F_{n,1}(x_2,x_3,\dots,x_d) +F_{n,2}(1,x_3,\dots,x_d)+\dots +F_{n,d}(1,x_2,\dots,x_{d-1}) \\
&\ & x_1\cdot x_3\cdots x_d\, h(x_1,1,\dots,x_d)  = \\
&\ & \lim_{n\to\infty}   F_{n,1}(1,x_3,\dots,x_d) +F_{n,2}(x_1,x_3,\dots,x_d)+\dots +F_{n,d}(x_1,1,\dots,x_{d-1}) \\
&\ &  \vdots  \\
&\ & x_1\cdot x_2\cdots x_{d-1}\, h(x_1,\cdots,x_{d-1},1) = \\
&\ & \lim_{n\to\infty}  F_{n,1}(x_2,\dots,x_{d-1},1) +F_{n,2}(x_1,x_3,\dots,x_{d-1},1)+\dots +F_{n,d}(x_1,\dots,x_{d-1})
\end{eqnarray*}
We can write this system of equations in the more compact form 
$$
\lim_{n\to\infty} \sum_{k=1}^d F_{n,k}(x_1,\dots,x_{j-1},\widehat{x}_j,x_{j+1},\dots x_{k-1},\widecheck{x}_k,x_{k+1},\dots,x_d)=
$$
$$
x_1\cdot x_2\cdots \widehat{x}_j\dots \cdot x_d\, h(x_1, x_2,\dots, \widehat{x}_j, \dots, x_d)
$$
for $j=1,2,\dots,d$ where $\widehat{x}_j$ means that $x_j=1$ and $\widecheck{x}_k$ means that the variable $x_k$ is missing. Now we sum up these equations for $j=1,2,\dots,d$. On the left hand side we recover the sum
$$
\sum_{k=1}^d F_{n,k}(x_1,\dots,x_{k-1},x_{k+1},\dots,x_d),
$$
which has the limit $x_1 x_2\cdots x_d\cdot h(x_1,x_2,\dots,x_d)$ as $n$ tends to infinity, and a sum $w_n(x_1,x_2,\dots,x_d)$ of polynomials each of them depending on $d-2$ out of the $d$ variables only.  
Obviously, we have 
$$
g(x_1,\cdots,x_d)=\lim_{n\to\infty} w_n(x_1,x_2,\dots,x_d)=
$$
$$
 \sum_{j=1}^d\left(h(x_1,\cdots,x_{j-1},1,x_j,\cdots,x_d)\prod_{k\neq j}x_k\right) - h(x_1,\cdots,x_{d})\prod_{k=1}^dx_k.
$$
As every term of the sum defining $w_n$ depends on $d-2$ variables only, it follows that with $\beta=(1,1,\dots,1,0)$ we have
$\Delta^{\beta}g=\lim_{n\to\infty}\Delta^{\beta}w_n=0$. As $g$ is a polynomial, this implies $\partial^{\beta} g=0$.
On the other hand, the assumption $h\neq 0$ implies that  
\begin{equation}\label{h}
h(x_1,\cdots,x_d)=\sum_{k=0}^sc_k(x_1,x_2,\cdots,x_{d-1})x_d^k
\end{equation}
for some natural number $s$ and polynomials $c_k$ in $d-1$ variables for $k=1,2,\dots,s$,  not all of them being identically zero.  A simple computation shows that
$$
0=\partial^{\beta}g=
\partial^{\beta}\left( x_1x_2\cdots x_{d-1}h(x_1,\cdots,x_{d-1},1)- x_1x_2\cdots x_{d}h(x_1,\cdots,x_{d}) \right),
$$
hence, substituting equation \eqref{h} into this formula, we obtain
\begin{eqnarray*}
0 &=&  \partial^{\beta} \left( \sum_{k=0}^s  x_1x_2\cdots x_{d-1}c_k(x_1,x_2,\cdots,x_{d-1})\right) \\
&\ & 
\ \ \ - \sum_{k=0}^s \partial^{\beta}\left (x_1x_2\cdots x_{d-1}c_k(x_1,x_2,\cdots,x_{d-1})\right)x_d^{k+1}.
\end{eqnarray*}
It follows that 
$$ 
\partial^{\beta}\left (x_1x_2\cdots x_{d-1}c_k(x_1,x_2,\cdots,x_{d-1})\right)=0, 
$$
for $k=0,1,\dots,s$. Now we assume that 
$$
c_k(x_1,\cdots,x_{d-1})=\sum_{|\alpha|\leq l}a_{\alpha,k}x_1^{\alpha_1}\cdots x_{d-1}^{\alpha_{d-1}}
$$ 
is not identically zero. Then 
\[
x_1x_2\cdots x_{d-1}c_k(x_1,x_2,\cdots,x_{d-1})= \sum_{|\alpha|\leq l}a_{\alpha,k}x_1^{\alpha_1+1}\cdots x_{d-1}^{\alpha_{d-1}+1},
\]
and 
$$
\partial^{\beta}\left (x_1x_2\cdots x_{d-1}c_k(x_1,x_2,\cdots,x_{d-1})\right )= 
$$
$$
\sum_{|\alpha|\leq l}(\alpha_1+1)\cdots (\alpha_{d-1}+1)a_{\alpha,k}x_1^{\alpha_1}\cdots x_{d-1}^{\alpha_{d-1}},
$$
which vanishes identically if and only if all coefficients $a_{\alpha,k}$ are zero. Hence the polynomial $h$ vanishes identically, which contradicts our assumptions. This proves the result for the very special case when $\Omega_V$ admits a decomposition of the form 
\[
\Omega_V=\bigcup_{k=1}^{d}Z_i,
\]
where  $Z_i=\{\alpha\in \mathbb{N}^d: \alpha_i\leq N_i\}$ for $i=1,2,\dots,d$.
\vskip.2cm

Now let $V$ be an arbitrary TDI-subspace of $\mathbb{R}[x]$ and let the sequence $(p_n)_{n\in\N}$ of elements of  $V$  converge pointwise to the polynomial $p$. Assume that 
$$
p(x)=\sum_{|\gamma|\leq \deg(p)}a_{\gamma}x^{\gamma}
$$ 
with $a_{\alpha}\neq 0$ for some  $\alpha=(\alpha_1,\cdots,\alpha_d)$ which is not in $\Omega_V$. We define
$$
\widetilde{\Omega} = \{ \beta: \beta_k < \alpha_k \text{ for at least one } k \}, \text{ and } \widetilde{V}=\mathbf{span}\{x^\alpha:\alpha\in \widetilde{\Omega} \}.
$$ 

Then $\widetilde{V}$ is a TDI-subspace of $\mathbb{R}[x]$ and  $\Omega_V \subseteq \widetilde{\Omega}$. Indeed, assuming that $\beta$ is in $\Omega_V$ with $\beta\not\in\widetilde{\Omega}$ gives $\beta_k \geq \alpha_k$ for all $k$, so $\alpha \leq \beta$, hence $\alpha$ is in $\Omega_V$, a contradiction.
\vskip.2cm

Obviously, $\widetilde{\Omega}=\bigcup_{k=1}^d Z_k$, where $Z_k = \{ \beta: \beta_k < \alpha_k \}$. Hence, by computations we made above, we  conclude that $p$ is in $\widetilde{V}$ which, by our construction, is impossible. This proves the theorem.

\end{proof}

 Note that without translation invariance  Theorem \ref{main} fails to hold. Indeed, if $\mathcal{P}$ denotes the set of prime numbers, M\"{u}ntz theorem \cite{A} guarantees that the linear span $V$ of the monomials $x\mapsto x^p$ with $p$ in  $\mathcal{P}\cup 2\mathcal{P}\cup\{0\}$ is dense in $C[a,b]$ over any interval $[a,b]$. In particular, we can find a sequence $(p_n)_{n\in\N}$ in $V$ such that 
 $$
 \|p_n-x^{100}\|_{\mathcal{C}[-n,n]}<\frac{1}{n}
 $$
holds for $n=1,2,\dots$. Obviously, $p_n$ converges pointwise to $x\mapsto x^{100}$, which is not in $V$. Furthermore, by Lemma \ref{nuevo} we have that $V$ is dilation invariant.
\vskip.2cm
 
Having proved our main result for TDI-spaces of polynomials a comparison between these spaces and the class of translation invariant spaces of polynomials should be of interest.  Here we present an extremal example. 

\begin{example}
Let $V$ be the set of all polynomials in two variables of the form $(x,y)\mapsto p(x+y)$, where $p$ is any polynomial in a single variable.
Then, by the Binomial Theorem, $V$ is 
translation invariant, and all monomials $x^ay^b$ with positive integers $a,b$ belong to $W$, the smallest TDI-space  which contains $V$. It follows that $W$ is the set of all polynomials in two variables, and the co-dimension of $V$ in $W$ is infinite. 
\end{example}
Finally, let us comment that in this example  a result analogous to Theorem \ref{main} can easily be proved for the space  $V$.


\begin{thebibliography}{99}
 \bibitem{AD} {\sc J. Acz\'{e}l, J. Dhombres, } \emph{Functional equations in several variables, } Encyclopedia of mathematics and its Applications \textbf{31}, Cambridge University Press, Cambridge, New-York, 2008.

\bibitem{A} {\sc J. M. Almira, } M\"{u}ntz type theorems I, Surveys in Approximation Theory, \textbf{3} (2007), 152--194. 


 \bibitem{Lang} {\sc S. Lang, } \emph{Algebra, } Graduate Texts in Mathematics, \textbf{211}, Springer-Verlag, New York,  2002


\bibitem{P}  {\sc A. Pinkus,  } TDI-Subspaces of $C(\mathbb{R}^d)$ and some Density Problems from Neural Networks, Journal of Approximation Theory \textbf{85} (3) (1996) 269--287. 

\bibitem{L1}
{\sc L. Sz\'{e}kelyhidi, }   \textit{Problem 1}, Problems and Remarks section of the  Forty-ninth International Symposium on Functional Equations, Graz-Mariatrost, June 19--26, 2011.

\bibitem{L2}
{\sc L. Sz\'{e}kelyhidi, }  \textit{Problem 4}, Problems and Remarks section of the  14th International Conference on Functional Equations and Inequalities, B\c{e}dlewo, September 11--17, 2011.

\end{thebibliography}
\end{document}